\documentclass[reqno,11pt]{amsart}
\usepackage[latin1]{inputenc}
\usepackage[T1]{fontenc}
\usepackage{lmodern}
\usepackage{a4wide}
\usepackage{amsmath,amsfonts,amsthm,amssymb}
\usepackage{enumitem}
\usepackage{mathtools}
\usepackage{color}
\usepackage{dsfont}
\usepackage{pgf,tikz}

\usepackage{hyperref}

\renewcommand\leq{\leqslant}
\renewcommand\geq{\geqslant}

\newcommand\be{\begin{equation}}
\newcommand\ee{\end{equation}}

\theoremstyle{plain}
\newtheorem{theorem}{\bf Theorem}[section]

\newtheorem{lemma}{\bf Lemma}[section]
\newtheorem{remark}{\bf Remark}[section]
\newtheorem{definition}{\bf Definition}[section]

\newtheorem{example}{\bf Example}[section]

\numberwithin{equation}{section}

\title[Some approximation results in Musielak-Orlicz spaces]{Some approximation results in Musielak-Orlicz spaces}

\author[A. Youssfi]{Ahmed Youssfi}
\address{University Sidi Mohamed Ben Abdellah,
	National School of Applied Sciences,
	P.O. Box 72 F\`es-Pricipale, Fez, Morocco}
\email{address:ahmed.youssfi@gmail.com ; ahmed.youssfi@usmba.ac.ma}

\author[Y. Ahmida]{Youssef Ahmida}
\address{University Sidi Mohamed Ben Abdellah,
	National School of Applied Sciences,
	P.O. Box 72 F\`es-Pricipale, Fez, Morocco}
\email{youssef.ahmida@usmba.ac.ma}

\begin{document}
\maketitle

\begin{abstract}
We give sufficient conditions for the continuity in norm of the
translation operator in the Musielak-Orlicz $L_{M}$ spaces. An application to the convergence in norm of approximate identities is given, whereby we prove density results of the smooth functions in $L_{M}$, in both modular and norm topologies. These density results are then applied to obtain basic topological properties.
\end{abstract}




\section{Introduction and statement of main results}\label{intro}
Classical Lebesgue and Sobolev spaces with constant exponent arise in the modeling of most materials with sufficient accuracy. For certain materials with inhomogeneities, for instance electrorheological fluids, this is not adequate, but rather the exponent should be able to vary. This leads to study those materials in Lebesgue and Sobolev spaces with variable exponent.\\
Historically, variable exponent Lebesgue spaces $L^{p(\cdot)}(\Omega)$, where $\Omega$ is an open subset of $\mathrm{R}^N$, was appeared in the literature for the first time in 1931 in a paper written by W. Orlicz \cite{O1931}.
The study of variable exponent Lebesgue spaces was then abandoned by Orlicz for the account of the theory of the function spaces $L_{M}(\Omega)$, built upon an $N$-function $M$, that now bear his name and which generalizes naturally the Lebesgue spaces with constant exponent. When we try to integrate both the two functional structures of variable exponent Lebesgue spaces and Orlicz spaces, we are led to the so-called Musielak-Orlicz spaces. This later functional structure was extensively studied from 1970's by the Polish school, notably by Musielak \cite{MJ} and Hudzik, see for instance \cite{HH4} and the references therein.
\par Here we are interested in establishing some basic approximation results in Musielak-Orlicz spaces with respect to the modular and norm
convergence, which then allow us to obtain some topological properties that constitute basic tools needed in the existence theory for partial differential equations involving nonstandard growths described in terms of Musielak-Orlicz functions. Such results requires to take into account the results earlier studied deeply in the monograph \cite{KrRu} and these presented in the pioneering work by Kov\'{a}cik and R\'{a}kosn\'{\i}k \cite{KR} concerning completeness, density, reflexivity and separability of variable exponent Lebesgue and Sobolev spaces.
\par Throughout this paper, we denote by $\Omega$ an open subset of $\mathrm{R}^N$, $N\geq 1$.
A real function $M$ : $\Omega\times[0,\infty)\to[0,\infty]$ is called a $\phi$-function, written $M\in\phi$, if
$M(x,\cdot)$ is a nondecreasing and convex function for all $x\in\Omega$ with $M(x,0)=0$, $M(x,s)>0$ for $s>0$, $M(x,s)\rightarrow\infty \mbox{ as }s\rightarrow\infty$ and $M(\cdot,s)$ is a measurable function for every $s\geq 0$.
A $\phi$-function is called $\Phi$-function, denoted by $M\in\Phi$, if furthermore it satisfies
$$
\lim_{s\to 0^+}\frac{M(x,s)}{s}=0 \mbox{ and } \lim_{s\to +\infty}\frac{M(x,s)}{s}=+\infty.
$$
Define $\overline{M}$ : $\Omega\times[0,\infty)\to[0,\infty]$ by
$$
\overline{M}(x,s)=\sup_{t\geq 0}\{st-M(x,t)\} \mbox{ for all } s\geq 0\mbox{ and all } x\in \Omega.
$$
It can be checked that $\overline{M}\in\phi$. The $\Phi$-function $\overline{M}$ is called the complementary function to $M$ in the sense of Young. Given $M\in\phi$, the Musielak-Orlicz space $L_M(\Omega)$ consists of the set of all measurable functions $u:\Omega\rightarrow \mathrm{R}$ such that $\int_\Omega M(x,|u(x)|/\lambda)dx<+\infty$ for some $\lambda>0$. Equipped with the so-called Luxemburg norm
$$
\|u\|_{L_M(\Omega)}=\inf\bigg\{\lambda>0: \int_\Omega M(x,|u(x)|/\lambda)dx\leq 1\bigg\},
$$
$L_M(\Omega)$  is a Banach space (see \cite[Theorem 7.7]{MJ}). It is a particular case of the so-called modular function spaces, investigated  by H. Nakano (see for instance \cite{Nak50}). We define $E_M(\Omega)$ as the subset of $L_M(\Omega)$ of all measurable functions $u:\Omega\to\mathrm{R}$ such that $\displaystyle\int_\Omega M(x,|u(x)|/\lambda)dx<+\infty$ for all $\lambda>0$.
\par Density result of smooth functions in Musielak-Orlicz-Sobolev spaces with respect to the modular topology  was claimed for the first time in \cite{BDV} in $\Omega=\mathrm{R}^N$ and then for a bounded star-shape Lipschitz domain $\Omega$ in \cite{BV}. The authors assumed that the $\Phi$-function $M$ satisfies, among others, the log-H\"{o}lder continuity condition, that is to say  there exists a constant $A>0$ such that for all $s\geq 1$,
\begin{equation}\label{loghtaha}
\frac{M(x,s)}{M(y,s)}\leq s^{-\frac{A}{\log|x-y|}},\; \forall x,y\in \Omega \mbox{ with } |x-y|\leq\frac{1}{2}.
\end{equation} 
Nonetheless, the proof involved an essential gap. The Jensen inequality was used for the infimum of convex functions, which obviously is not necessarily convex.


Unlike the classical Orlicz spaces, the spatial dependence of the $\phi$-function $M$ does not allow, in general, to bounded functions to belong to Musielak-Orlicz spaces. An extra additional hypothesis is needed. In the sequel, we make use of the following local integrability.
\begin{definition}
	We say that $M\in\Phi$ is locally integrable, if for any constant number $c>0$ and for every compact set $K\subset\Omega$ we have
	\begin{equation}\label{incBM}
	\int_K M(x,c)dx<\infty.
	\end{equation}
\end{definition}
Let us note that if $M\in\phi$ (resp. $\overline{M}\in\phi$) satisfies $\lim_{s\to \infty} \mathrm{ess\,inf}_{x\in\Omega}\frac{M(x,s)}{s}=+\infty$ (resp. $\lim_{s\to \infty} \mathrm{ess\,inf}_{x\in\Omega}\frac{\overline{M}(x,s)}{s}=+\infty$), then $\overline{M}$ (resp. $M$) satisfies (\ref{incBM}) not only for compact subsets $K\subset\Omega$ but for all mesurable subsets of $\Omega$  having finite Lebesgue measure .
\par From now on, $\mathcal{B}_c(\Omega)$ will stands for the set of bounded functions compactly supported in $\Omega$ and $\mathcal{C}^\infty_0(\Omega)$ will denotes the set of infinitely differentiable functions compactly supported in $\Omega$.\\ 
The condition (\ref{incBM}) ensures that the set $\mathcal{B}_c(\Omega)$ is contained in $E_M(\Omega)$. Incidentally, the functions essentially bounded do not belong necessary to $E_M(\Omega)$ even if (\ref{incBM}) is filled. Observe that if $M(x,|s|)=M(|s|)$ is independent of $x$ (that is 
$M(\cdot)$ is a continuous, nondecreasing and convex function with $M(0)=0$, $M(s)>0$ for $s>0$, $M(s)\to\infty$ as $s\to\infty$,
$\lim_{s\to 0^+}\frac{M(s)}{s}=0$ and $\lim_{s\to +\infty}\frac{M(s)}{s}=+\infty$), inequality (\ref{incBM}) is obviously satisfied. Here, we do not need to assume the condition (\ref{loghtaha}) but we only assume (\ref{incBM}). Note in passing that if $x\mapsto M(x,s)$ is a continuous function on $\Omega$, then (\ref{incBM}) holds trivially.
\par The main results  we obtain in this paper cover those already known in the frameworks of Orlicz and variable exponent Lebesgue spaces and are contained in the four following theorems that we prove in Section \ref{sec4}.
\begin{theorem}\label{th1.1.1}
	Let $M\in\Phi$ satisfy (\ref{incBM}), then
	\begin{enumerate}
		\item $\mathcal{C}^\infty_0(\Omega)$ is dense in $E_M(\Omega)$ with respect to the strong topology in $E_M(\Omega)$.
		\item $\mathcal{C}^\infty_0(\Omega)$ is dense in $L_M(\Omega)$ with respect to the modular topology in $L_M(\Omega)$.
	\end{enumerate}
\end{theorem}
\begin{theorem}\label{th1.1.2}
	Assume that $M\in\Phi$ satisfies (\ref{incBM}). Then, the space $E_M(\Omega)$ is separable.
\end{theorem}
\begin{theorem}\label{thE1}
	Let $M\in\Phi$ satisfy (\ref{incBM}) and let $\overline{M}$ stands for it's complementary function. Then, the dual space  $(E_M(\Omega))^\prime$ of $E_M(\Omega)$ is isomorphic to $L_{\overline{M}}(\Omega)$, denote $(E_M(\Omega))^\prime\simeq L_{\overline{M}}(\Omega)$.
\end{theorem}
Another aspect of duality in Musielak spaces can be found in \cite{KK}. We present here a simple proof involving the results obtained. A special class of $\Phi$-functions is the following.
\begin{definition}
	We say that $M\in\Phi$ satisfies the $\Delta_2$-condition, written $M\in\Delta_2$, if there exist a constant $k>0$ and a nonnegative function $h\in L^1(\Omega)$ such that
	\begin{equation}\label{delta2}
	M(x,2t)\leq kM(x,t)+h(x),
	\end{equation}
	for all $t\geq 0$ and for almost every  $x\in\Omega$.
\end{definition}
\begin{remark}
	The hypothesis (\ref{delta2}) seems to be natural in the sense that if the $\Phi$-function $M$ is independent of the variable $x$, i.e. $M$: $[0,\infty)\to [0,\infty]$ is an $N$-function (see \cite{AA}), it is equivalent to the formulation of the $\Delta_2$-condition considered in the Orlicz spaces setting. That is there exist two constants $c>0$ and $t_0>0$ such that
	\begin{equation}\label{delta2_or}
	M(2t)\leq cM(t)
	\end{equation}	
	for every $t\geq t_0$ or $t\geq 0$ according to whether  $\Omega$ has finite Lebesgue measure or not successively. Indeed, assume that $|\Omega|<+\infty$, where from now on $|E|$ stands for the Lebesgue measure of a subset $E$ of $\mathrm{R}^N$. Suppose that (\ref{delta2}) holds but (\ref{delta2_or}) is not true. Thus, there exists a sequence $t_n\geq t_0$ such that $t_n\to+\infty$ satisfying $M(2t_n)>2^nM(t_n)$. Then, by (\ref{delta2}) we have
	$$
	\int_{\Omega}h(x)dx\geq (2^n-k)M(t_n)|\Omega|.
	$$
	The passage to the limit as $n$ tends to infinity yields a contradiction with the fact that $h\in L^1(\Omega)$. Conversely, assume that (\ref{delta2_or}) is fulfilled. The function $f$ defined by $f(t)=M(2t)-cM(t)$ is continuous  and so we get $0\leq \max_{0\leq t\leq t_0}f(t)=C<+\infty$. Therefore, we obtain
	$$
	M(2t)\leq cM(t)+C, \mbox{ for all } t\geq0.
	$$
	Hence, we can choose $h(x)=C$. Suppose now that $|\Omega|=+\infty$. Assume that (\ref{delta2}) holds but (\ref{delta2_or}) is not true. So that for $c=k$, there is $s\geq 0$ satisfying $M(2s)>kM(s)$. From (\ref{delta2}) we can write $h(x)\geq M(2t)-kM(t)$ for all $t\geq 0$. Substituting $t$ by $s$ and then integrating over $\Omega$ one has
	$$
	\int_{\Omega}h(x)dx\geq (M(2s)-kM(s))|\Omega|=+\infty,
	$$
	which contradicts the fact that $h\in L^1(\Omega)$. Conversely, if (\ref{delta2_or}) is true then (\ref{delta2}) holds for $h=0$.
\end{remark}
\begin{theorem}\label{th1.2.2}
	Let $M,$ $\overline{M}\in\Phi$ be a  pair of complementary $\Phi$-functions satisfying both (\ref{incBM}) and the $\Delta_2$-condition. Then, the Musielak-Orlicz space $L_M(\Omega)$ is reflexive.
\end{theorem}
\par Theorem \ref{th1.1.1} is a unified generalization of the approximation results known in Lebesgue spaces $L^p(\Omega)$, $1<p<\infty$ and Orlicz spaces. The approach, now classical, is based upon reducing the study to continuous functions compactly supported in $\Omega$ and then using imbedding theorems and a sequence of mollifiers to conclude, see for instance \cite[Corollary 2.30 and Theorem 8.21]{AA}. This classical approach is based on the fact that the translation operator $u(\cdot+h)$ is continuous in norm when $h$ tends to zero. This fails to hold in Musielak-Orlicz spaces (see Remark \ref{remark1} below and \cite[Proposition 3.6.1]{DHHM}). Consequently, we can not approximate in general the identities for a given function.
\par In the framework of variable exponent Lebesgue spaces $L^{p(\cdot)}(\Omega)$, Kov\'{a}cik and R\'{a}kosn\'{i}k \cite[Theorem 2.11]{KR} proved first the density of the set $\mathcal{C}^\infty_0(\Omega)$ of infinitely differentiable functions compactly supported in $\Omega$, provided only that the variable exponent $p(\cdot)\in L^{\infty}(\Omega)$. Their idea consists in showing successively that the set of essentially bounded functions $L^\infty (\Omega)\cap L^{p(\cdot)}(\Omega)$ is dense in $L^{p(\cdot)}(\Omega)$ and by means of Luzin's theorem the subset of continuous functions $\mathcal{C}(\Omega)\cap L^{p(\cdot)}(\Omega)$ is dense in $L^{p(\cdot)}(\Omega)$, which finally lead to the density of the set $\mathcal{C}^\infty_0(\Omega)$ in $L^{p(\cdot)}(\Omega)$.
\par In Musielak-Orlicz spaces, the situation is more complicated. At first,
we note that although the assumption (\ref{incBM}) is satisfied, the inclusion $L^\infty(\Omega)\subset L_{M(\cdot,\cdot)}(\Omega)$ does not hold true in general even if $\Omega$ is an open of $\mathrm{R}^N$ with finite Lebesgue measure. Moreover, the Musielak space $E_M(\Omega)$ does not contain, in general, the set of smooth functions $\mathcal{C}^\infty_0(\Omega)$. Secondly, the use of the similar idea to that of Cruz-Uribe and Fiorenza \cite{CF} will requires additional assumptions. Thirdly, in contrast to what is mentioned above, the translation operator is not acting, in general, between Musielak-Orlicz spaces (see \cite[example 2.9 and Theorem 2.10]{KR} and \cite[Proposition 3.6.1]{DHHM}). For these reasons and for the best of our knowledge, it is not possible to obtain approximation results using classical ideas. The approach we use is based on the local integrability assumption (\ref{incBM}) and consists in beginning by proving
the density of smooth functions $\mathcal{C}^\infty_0(\Omega)$ in $\mathcal{B}_c(\Omega)$ with respect to the norm in $L_M(\Omega)$, (see Lemma \ref{lem3.4}), and then the density of bounded functions compactly supported in norm in $E_M(\Omega)$ and in modular in $L_M(\Omega)$, (see Lemma \ref{th1.1.8}), which allow us to get the density of smooth functions $\mathcal{C}^\infty_0(\Omega)$ in $E_M(\Omega)$ and $L_M(\Omega)$ with respect to the norm and modular convergence respectively. The idea we use in this paper is essentially based upon using the fact that for a function $u\in\mathcal{B}_c(\Omega)$, the translation operator $u(\cdot+h)$ is continuous with respect to the norm $\|\cdot\|_{L_{M}(\Omega)}$ as $h$ tends to $0$ (see Lemma \ref{lem6.1} below). This constitutes the main reason why we introduce the space $\mathcal{B}_c(\Omega)$.

\par In the framework of classical Lebesgue or Orlicz spaces, the separability of the closure of bounded functions compactly supported in $\overline{\Omega}$ is well-known, see for instance \cite[Theorem 2.21 and Theorem 8.21]{AA}, while for bounded variable exponent spaces one can see \cite[corollary 2.12 ]{KR}.
Here, using the density results obtained in Theorem \ref{th1.1.1} we prove the separability of $E_M(\Omega)$.
Our proof is totally different from that given by Musielak \cite[Theorem 7.10]{MJ}, since the author used the density of simple functions assuming that
the $\Phi$-function $M$ satisfies the local integrability condition on a measure space $(\Omega,\Sigma,\mu)$ where $\mu$ is a positive complete  measure, that is to say $\int_D M(x,s)d\mu<\infty$ for every $s>0$ and $D\in\Sigma$  with $\mu(D)<\infty$,
(see \cite[Definition 7.5]{MJ}), which is stronger than the condition (\ref{incBM}) when we limit ourselves to the Borel tribe.
\par In the Orlicz spaces, the density of simple functions, denoted $\mathcal{S}$, in $E_M(\Omega)$ is an important step in the proof of the duality result (see for instance \cite[Theorem 8.19]{AA}). In the Musielak-Orlicz spaces such approximation result needs to assume a local integrability condition (see \cite[Theorem 7.6]{MJ}) which allows us to get the inclusion $\mathcal{S}\subset E_M(\Omega)$.\\
In Theorem \ref{thE1} we use the weaker condition (\ref{incBM})
and act otherwise in the prove by using the density of the set $S_c$ of simple functions compactly supported in $\Omega$ ( see Lemma \ref{th1.2.1} below) instead of simple functions only. Related to this topic, one can also consult \cite{KK}.
\par The paper is organized as follow, in Section \ref{sec1} we give some properties of Musielak-Orlicz spaces where the most of them can be found in the standard monograph of J. Musielak \cite{MJ}. In Section \ref{sec2} we prove the continuity in norm of the translation. In Section \ref{sec3}, we give some auxiliary lemmata needed to prove the main results. The proof of the main results are presented in Section \ref{sec4}.
\section{Background}\label{sec1}
Here we give some definitions and properties that concern Musielak-Orlicz spaces. Observe first that equivalently a $\Phi$-function $M$ can be represented as (see \cite[Theorem 13.2]{MJ})
$$
M(x,t)=\int_{0}^{t}a(x,s)ds, \mbox{ for }t\geq 0,
$$
where $a(x,\cdot)$ is a right continuous and increasing function, $a(x,s)>0$ for $s>0$, $a(x,0)=0$, $a(x,s)\rightarrow\infty$ as $s\rightarrow \infty$
for every $x\in \Omega$.
The complementary of a $\Phi$-function $M$ (see \cite[Definition 13.4]{MJ}) can be also expressed as
$$
\overline{ M}(x,t)=\int_{0}^{t}a^*(x,s)ds, \mbox{ for }t\geq 0,
$$
where $a^*(x,s)=\sup\{v, a(x,v)\leq s\}$. Moreover, we have the Young inequality
\begin{equation}\label{younginequality}
uv\leq M(x,u)+ \overline{M}(x,v),\quad \forall u,v\geq 0, \forall x\in\Omega,
\end{equation}
which reduces to an equality when $v=a(x,u)$ or $u=a^{\ast}(x,v)$.
It's easy to check that
\begin{equation}\label{eq1.7.1}
\|u\|_{L_M(\Omega)}\leq 1\Leftrightarrow \int_\Omega M(x,|u(x)|)dx\leq 1.
\end{equation}
We also have the following H\"{o}lder inequality (see \cite[Theorem 13.13]{MJ})
\begin{equation}\label{eq1.1.4}
\int_\Omega |u(x)v(x)| dx\leq\|u\|_{M}\|v\|_{L_{\overline{M}}(\Omega)}
\end{equation}
for all $u\in L_M(\Omega)$ and $v\in L_{\overline{M}}(\Omega)$, where
$\|u\|_{M}=\displaystyle\sup_{\|v\|_{L_{\overline{M}}}\leq1}\int_\Omega |u(x)v(x)|dx$ is the Orlicz norm.
We denote by $M^{-1}$ the inverse of the $\Phi$-function $M$ with respect to its second argument defined as follows
$$
M^{-1}(x,t)=\inf\{s\geq 0, M(x,s)\geq t\}.
$$
Thus, it follows
\begin{equation}\label{inversefunction1}
M^{-1}(x,M(x,s))=M(x,M^{-1}(x,s))=s.
\end{equation}
A sequence $\{u_n\}$ is said to converge to $u$ in $L_M(\Omega)$ in the modular sense, if there exists $\lambda>0$ such that
$$
\int_{\Omega}M\Big(x,\frac{|u_n(x)-u(x)|}{\lambda}\Big)dx\rightarrow 0, \mbox{ as } n\rightarrow \infty.
$$
We say that $\{u_n\}$ converges to $u$ in norm in $L_M(\Omega)$,
if $\|u_n-u\|_{L_M(\Omega)}\rightarrow 0$ as $n\rightarrow \infty$.
\begin{definition}
	A $\Phi$-function $M$ is said to satisfy the weak $\Delta_2$-condition, written $M\in\Delta_{2,w}$, if for every $u_n\in L_{M}(\Omega)$
	$$
	\int_{\Omega}M(x,u_n)dx\rightarrow 0
	\mbox{ implies }\int_{\Omega}M(x,2u_n)dx\rightarrow 0 \mbox{ as } n\rightarrow \infty.
	$$
\end{definition}
\begin{lemma}
	Let $M\in\Phi$. Then the norm convergence and the modular convergence are equivalent if and only if $M\in\Delta_{2,w}$.
\end{lemma}
\begin{proof}
	Suppose that the norm convergence and the modular convergence are equivalent and let $u_n\in L_{M}(\Omega)$ be such that $\displaystyle\int_{\Omega}M(x,u_n)dx\rightarrow 0$. Then, $\|u_n\|_{L_{M}(\Omega)}\to0$ as $n\rightarrow \infty$. By the definition of the norm, we can get
	$$
	\int_{\Omega}M(x,2u_n)dx\leq2\|u_n\|_{L_{M}(\Omega)}.
	$$
	Thus, $\displaystyle\int_{\Omega}M(x,2u_n)dx\rightarrow 0$ as $n\rightarrow\infty$.\\
	Conversely, let $u_n\in L_{M}(\Omega)$ and suppose that $\displaystyle\int_{\Omega}M\Big(x,\frac{|u_n|}{\lambda}\Big)dx\rightarrow 0$ for some $\lambda>0$. We shall prove that $\displaystyle\int_{\Omega}M(x,\mu |u_n|)dx\rightarrow 0$ as $n\rightarrow\infty$ for every $\mu>0$. Let $m\in\mathrm{N}$ be such that $\lambda\mu\leq 2^m$. Then, we can write
	$$
	\int_{\Omega}M(x,\mu |u_n|)dx\leq\frac{\lambda\mu}{2^m}\int_{\Omega}M\Big(x,2^m\frac{|u_n|}{\lambda}\Big)dx\rightarrow 0 \mbox{ as }n\rightarrow\infty.
	$$
\end{proof}
\begin{lemma}
	Let $M\in\Phi$. If $M\in\Delta_2$, then the norm convergence and the modular convergence are equivalent.
\end{lemma}
This lemma was proved in \cite[Theorem 8.14]{MJ} using the local integrability (see \cite[Definition 7.5]{MJ}). This local integrability is not necessary, we give here a proof based on Vitali's theorem. 
\begin{proof}
	We only need to prove that the modular convergence implies the norm convergence, the converse is an easy task. Let $\{u_n\}$ be a sequence of functions belonging to $L_M(\Omega)$ such that
	$\displaystyle\int_\Omega M(x,u_n(x)/\lambda)\rightarrow 0$ as $n\to\infty$, for some $\lambda>0$. Thus, $M(x,u_n(x)/\lambda)\rightarrow 0$ strongly in $L^1(\Omega)$. Hence, for a subsequence still indexed by $n$, we can assume that $u_n\to u$ a.e. in $\Omega$. Let $p$ be a fixed integer, by the $\Delta_2$-condition we can write
	$$
	M(x,2^p u_n(x)/\lambda)\leq k^p M(x, u_n(x)/\lambda)+ (k^{p-1}+\cdots+k+1) h(x).
	$$
	Therefore, by Vitali's theorem we get $\displaystyle\lim _{n\rightarrow\infty}\int_\Omega M(x,2^p u_n)= 0$. For an arbitrary $\lambda>0$, there exists $m$ such that $\lambda\leq 2^m$. Then we can write
	$$
	\int_\Omega M(x,\lambda u_n)dx\leq \frac{\lambda}{2^m} \int_\Omega M(x,2^m u_n)dx\rightarrow 0 \mbox{ as } n\rightarrow \infty,
	$$
	which gives $\|u_n\|_{L_M(\Omega)}\rightarrow 0$ as $n\rightarrow \infty$.
\end{proof}
\par Note that $E_M(\Omega)$  is a closed subset of $L_M(\Omega)$. Indeed, let $\{u_n\}\subset E_M(\Omega)$ such that
$u_n\rightarrow u\in L_M(\Omega)$. For any $\lambda>0$, we have $\displaystyle\int_\Omega M(x,2\lambda|u_n-u|)dx\rightarrow 0$. This implies
$M(x,2\lambda|u_n-u|)\rightarrow 0$ in $L^1(\Omega)$. So that, there is $h\in L^1(\Omega)$ such that $M(x,2\lambda|u_n-u|)\leq h$. Hence,
$$
\lambda|u_n(x)-u(x)|\leq\frac{1}{2}M^{-1}(x,h(x)),
$$
which yields
$$
\lambda|u(x)|\leq\lambda|u_n(x)|+\frac{1}{2}M^{-1}(x,h(x)).
$$
By the convexity of $M$, we get
$$
M(x,\lambda|u(x)|)\leq \frac{1}{2} M(x,2\lambda|u_n(x)|)+\frac{1}{2}h(x).
$$
Thus
$$
\int_\Omega M(x,\lambda|u(x)|)dx\leq \frac{1}{2}\int_\Omega M(x,2\lambda|u_n(x)|)dx+\frac{1}{2}\int_\Omega h(x)dx<\infty.
$$
So that $u\in E_M(\Omega)$.
\begin{lemma}\label{delta2equivalence}
	Let $M\in\phi$, then the following assertions are equivalent
	\begin{enumerate}
		\item [i)]  $E_M(\Omega)=L_M(\Omega)$.
		\item [ii)] $M\in\Delta_2$.
	\end{enumerate}
\end{lemma}
\begin{proof}
	
	i)$\Rightarrow$ ii) For any $u\in L_M(\Omega)$ we have $2u\in L_M(\Omega)$. Then it follows
	$$
	L_{M}(\Omega)\subset L_{\widetilde{M}}(\Omega),\; \mbox{ where }\; \widetilde{M}(x,u)=M(x,2u).
	$$
	Therefore, by \cite[Theorem 8.5 (b)]{MJ} there exist a constant $k>0$ and a nonnegative function $h\in L^1(\Omega)$ such that
	$$
	\widetilde{M}(x,u)=M(x,2u)\leq kM(x,u)+h(x).
	$$
	
	ii)$\Rightarrow$ i) For $u\in L_M(\Omega)$ there exists $\lambda>0$ such that $\displaystyle\int_\Omega M(x,|u(x)|/\lambda)dx<\infty$. We shall prove that for every $\mu>0$ we have $\displaystyle\int_\Omega M(x,|u(x)|/\mu)dx<\infty$. Indeed, there exists an integer $m$ such that  $\frac{\lambda}{\mu}\leq 2^m$. Thus, we can write
	$$
	\begin{array}{lll}\displaystyle
	\int_\Omega M(x,|u(x)|/\mu)dx \\
	\leq\displaystyle\frac{\lambda}{2^m\mu}\int_\Omega  M\Big(x,\frac{2^m|u(x)|}{\lambda}\Big)dx\\
	\leq\displaystyle\frac{\lambda}{2^m\mu}  \Bigg(k^m\int_\Omega M\Big(x,\frac{|u(x)|}{\lambda}\Big)dx+(k^{m-1}+\cdots+k+1)\int_\Omega h(x)dx \Bigg)
	<\infty.
	\end{array}
	$$
\end{proof}
Let $J$ stands for the Friedrichs mollifier kernel
defined on $\mathrm{R}^{N}$ by
$$
J(x)= ke^{-\frac{1}{1-\|x\|^2}}  \hbox{ if } \|x\|<1 \mbox{ and } 0  \hbox{ if } \|x\|\geq 1,
$$
where $k>0$ is such that $\int_{\mathrm{R}^{N}}J(x)dx=1$. For $\varepsilon>0,$ we  define $J_\varepsilon(x)=\varepsilon^{-N}J(x\varepsilon^{-1})$ and $u_\varepsilon=J_\varepsilon\ast u$ by
\begin{equation}\label{eq11}
u_\varepsilon(x)=\int_{\mathrm{R}^N}J_\varepsilon(x-y)u(y)dy
=\int_{B(0,1)}u(x-\varepsilon y)J(y)dy.
\end{equation}
\begin{example} As examples of $\Phi$-functions, we give
	$$
	\begin{array}{lll}
	M_1(x,s)= |s|^{p(x)}, \;1<p(\cdot)<\infty,\\  M_2(x,s)=|s|^{p(x)}\log(e+|s|), \; 1<p(\cdot)<\infty,\\
	M_3(x,s)= \frac{1}{p(x)}[(1+|s|^2)^{\frac{p(x)}{2}} -1], \;1<p(\cdot)<\infty,\\ M_4(x,s)= |s|^{p} +a(x)|s|^{q}, \;1< p< q, 0\leq a(\cdot)\in L^{1}_{loc}(\Omega),\\
	M_5(x,s)= e^{|s|^{p(x)}} -1 ,\;1<p(\cdot)<\infty, \\
	\end{array}
	$$
	In the last two decades, a wide literature was grown up for the account of the first $\Phi$-function $M_1$ which leads to variable exponent Lebesgue spaces. The $\Phi$-function $M_2$ arises particularly in plasticity when $p(\cdot)$ is a constant number. Observe that $M_4\in\Delta_2$ and if $p^+:=ess\sup_{x\in\Omega}p(x)<+\infty$, the $\Phi$-functions $M_i$, $1\leq i\leq3$, satisfy the $\Delta_2$-condition while it is no longer the case for $M_5$.
\end{example}
\begin{remark} Going back to the local integrability (\ref{incBM}).\\
	If $p^+={\rm ess}\sup_{x\in \Omega}p(x)<+\infty$, it is obvious to check that the $\Phi$-functions $M_i$, $i\in\{1,2,3,5\}$, above satisfy the assumption (\ref{incBM}). Furthermore, if $a\in L^{1}_{loc}(\Omega)$ then the $\Phi$-function $M_4$ satisfies (\ref{incBM}). 
\end{remark}
The equivalence between Orlicz and Luxemburg norms is well-known in the Orlicz spaces setting \cite[Theorem 3.8.5]{KJF}, while in the Musielak-Orlicz framework this result was proved by Musielak \cite[Theorem 13.11]{MJ} using a local integrability condition upon measurable sets with finite measure.
\begin{lemma}\label{lem1.1.6}
	Let $M\in\Phi$ satisfy (\ref{incBM}). Then, for all $u\in L_M(\Omega)$
	\begin{equation}\label{eq1.3.3}
	\|u\|_{L_M(\Omega)}\leq\|u\|_M\leq 2\|u\|_{L_M(\Omega)}
	\end{equation}
\end{lemma}
\begin{proof}
	The inequality in the right-hand side is an easy consequence of the Young inequality. We only need to prove the left-hand side inequality. To this end, it is sufficient to prove that
	$$
	\int_\Omega M(x, |u(x)|/\|u\|_M)dx\leq 1.
	$$
	This can be done by using (\ref{incBM}) and following exactly the line of \cite[Lemma 3.7.2]{KJF}.
\end{proof}
\section{Properties of the translation operator}\label{sec2}
\subsection{M-mean continuity}
For $h\in\mathrm{R}^N$, let $\tau_h u$ stands for the translation operator defined by
$$
\tau_h u(x)=
\left\{
\begin{array}{lll}
u(x+h)& \mbox{ if } x\in \Omega \mbox{ and } x+h\in \Omega, \\
0 &\mbox{ otherwise in }\mathrm{R}^N.
\end{array}
\right.
$$
If the function $u$ has a compact support, $\tau_h u$ is well-defined provided that $h<dist(supp\; u,\partial\Omega)$.
In general, if $u\in E_M(\Omega)$ we can't expect that $\tau_h u$ belongs to $E_M(\Omega)$ (see \cite[Example 2.9 and Theorem 2.10]{KR}). In the following lemma, we prove that the translation operator acts on the set of bounded functions compactly supported in $\Omega$. In the case where $M(x,t)=|t|^{p(x)}$, a similar result was proved in \cite[p.261]{CF} by using the continuous imbedding between variable exponent Lebesgue spaces.
Unfortunately, this result is not true in general in variable Lebesgue spaces \cite[Proposition 3.6.1]{DHHM} unless the exponent is constant.
\begin{lemma}\label{lem6.1}
	Let $M\in\Phi$ satisfy (\ref{incBM}). Then, any $u\in \mathcal{B}_c(\Omega)$ is M-mean continuous, that is to say for every $\varepsilon>0$ there exists a $\eta=\eta(\varepsilon)>0$ such that for $h\in\mathrm{R}^N$ with $|h|<\eta$ we have
	$$
	\|\tau_h u-u \|_{L_M(\Omega)}<\varepsilon.
	$$
\end{lemma}
\begin{proof}
	For $u\in\mathcal{B}_c(\Omega)$, let $supp u=U\subset B_{R}\cap\Omega$, where by $B_{R}$ we denote a ball with radius $R>0$.
	Let $h\in\mathrm{R}^N$ with $|h|<\min(1,dist(U,\partial\Omega))$.  We have $supp\,\tau_h u\subset B:=B_{R+1}\cap\Omega$. Let us define $\overline{B}=\overline{B}_{R+1}\cap\Omega$ where $\overline{B}_{R+1}$ stands for the closed ball with radius $R+1$. Thanks to (\ref{incBM}), for any constant number $C>0$ one has $M(x,C)\in L^{1}(\overline{B})$. Therefore, for arbitrary $\varepsilon>0$, there is $\nu>0$ such that for all measurable subset $\Omega^\prime\subset\overline{B}$
	\begin{equation}\label{eq6.2}
	\int_{\Omega^\prime}M(x,C)dx<\frac{\varepsilon}{2}, \mbox{ whenever }|\Omega^\prime|<\nu.
	\end{equation}
	For this $\nu$, there exists $\rho\in(0,1)$ such that $|H_\rho|<\frac{\nu}{4}$ where
	$$
	H_\rho=\{x\in B: dist (x,\partial B)\leq\rho\}.
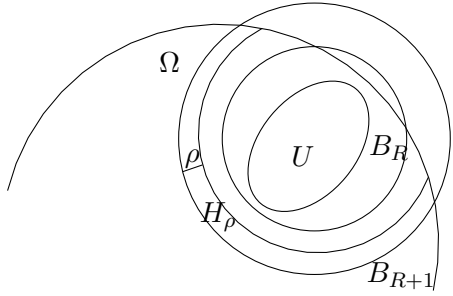
	$$
	\definecolor{xdxdff}{rgb}{0.49019607843137253,0.49019607843137253,1.}
	\definecolor{qqqqff}{rgb}{0.,0.,1.}
	\begin{center}
		\begin{tikzpicture}
		\draw [shift={(2.103496332518336,2.1273227383863027)}] plot[domain=-0.22101557782652304:2.9010177971452045,variable=\t]({1.*2.907220894342422*cos(\t r)+0.*2.907220894342422*sin(\t r)},{0.*2.907220894342422*cos(\t r)+1.*2.907220894342422*sin(\t r)});
		\draw (3.95,3.69) node[anchor=north west] {$B_R$};
		\draw [rotate around={50.38931175997341:(3.28,3.41)}] (3.28,3.41) ellipse (0.9910105994563201cm and 0.6444392975562363cm);
		\draw(3.36,3.51) circle (1.225887433657757cm);
		\draw (2.92,3.55) node[anchor=north west] {$U$};
		\draw (1.16,4.8) node[anchor=north west] {$\Omega$};
		\draw(3.36,3.51) circle (1.8054362353736002cm);
		\draw (3.92,2.0) node[anchor=north west] {$B_{R+1}$};
		\draw (1.5,3.5) node[anchor=north west] {$\rho$};
		\draw (1.7,2.8) node[anchor=north west] {$H_\rho$};
		\draw [shift={(3.4007292160767,3.576779380526849)}] plot[domain=2.0533407331940277:5.911369247541888,variable=\t]({1.*1.5836882111269457*cos(\t r)+0.*1.5836882111269457*sin(\t r)},{0.*1.5836882111269457*cos(\t r)+1.*1.5836882111269457*sin(\t r)});
		\draw (1.6084695767616795,3.0721173941904194)-- (1.8721969164861136,3.1624638495956776);
		\end{tikzpicture}
	\end{center}
	Define $U_\rho=B\setminus H_\rho$. Since $u$ is measurable on $U_\rho$, Luzin's theorem ensures that for $\nu>0$ there  exists a closed set $F_{1,\nu}\subset U_\rho$ such that the restriction of $u$ to $F_{1,\nu}$ is continuous and $|U_\rho \setminus F_{1,\nu}|<\frac{\nu}{4}$. We then have, $|B\setminus F_{1,\nu}|<\frac{\nu}{2}$. The function $u$ is uniformly continuous on the compact set $F_{1,\nu}$. It follows that for $\varepsilon>0$, there exists an $\eta\in(0,\rho)$ such that for all $x$, $x+h\in F_{1,\nu}$ one has
	\begin{equation}\label{cont}
	|h|<\eta \Rightarrow |u(x+h)-u(x)|<\frac{\varepsilon}{2\Big(\displaystyle\int_{\overline{B}}M(x,1)dx+1\Big)}.
	\end{equation}
	Define the two sets
	$$
	F_{2,\nu}=\{x\in U, x+h\in F_{1,\nu}\} \mbox{ and } F_{\nu}=F_{1,\nu}\cap F_{2,\nu}.
	$$
	The set $F_{\nu}$ is a closed subset of $\Omega$. In addition, we have $|B\setminus F_{\nu}|<\nu$. Indeed, since the Lebesgue measure is invariant by translation we get $|B\setminus F_{1,\nu}|=|B\setminus F_{2,\nu}|$. Therefore,
	\begin{equation}\label{eq6.3}
	|B\setminus F_{\nu}|=|(B\setminus F_{1,\nu})\cup(B\setminus F_{2,\nu})|\leq|B\setminus F_{1,\nu}|+|B\setminus F_{2,\nu}|<\nu.
	\end{equation}
	If $x\notin B$ then for $|h|<\eta$ we have
	$x+h\notin B_{R}\cap\Omega$. If not, we will get $x\in B$ which contradicts the fact that $x\notin B$. Hence, we obtain
	\begin{equation}\label{eq6.4}
	\begin{array}{lll}\displaystyle
	\int_{\Omega}M(x,|\tau_hu(x)-u(x)|)dx &=\displaystyle\int_{B}M(x,|\tau_hu(x)-u(x)|)dx \\
	&=\displaystyle\int_{B\cap F_{\nu}}M(x,|\tau_hu(x)-u(x)|)dx\\
	&=+\displaystyle\int_{B\setminus F_{\nu}} M(x,|\tau_hu(x)-u(x)|)dx.
	\end{array}
	\end{equation}
	By (\ref{cont}) the first term in the right-hand side can be estimated as
	$$
	\int_{B\cap F_{\nu}}M(x,|\tau_hu(x)-u(x)|)dx\leq \int_{B\cap F_{\nu}}M\bigg(x,\frac{\varepsilon}{2\Big(\displaystyle\int_{\overline{B}}M(x,1)dx+1\Big)}\bigg)dx<\frac{\varepsilon}{2}.
	$$
	As regards the second term in the right-hand side of (\ref{eq6.4}), we use the fact that $u\in\mathcal{B}_c(\Omega)$ is bounded by a constant number $c>0$ and then (\ref{eq6.2}) to obtain
	\begin{equation}\label{eq1.3.8}
	\int_{B\setminus F{\nu}}M(x,|\tau_hu(x)-u(x)|)dx\leq \int_{B\setminus F{\nu}}M(x,2c)dx\leq\frac{\varepsilon}{2}.
	\end{equation}
	Puting together (\ref{eq6.4}) and (\ref{eq1.3.8}), we get
	$$
	\forall \varepsilon>0, \exists \eta>0: |h|<\eta \Rightarrow \int_{\Omega}M(x,|\tau_hu(x)-u(x)|)dx<\varepsilon.
	$$
	Let $\delta>0$ be arbitrary but fixed. As $u/\delta\in\mathcal{B}_c(\Omega)$, we get
	$$
	\exists \eta>0 : |h|<\eta \Rightarrow \int_{\Omega}M\Big(x,\frac{|\tau_hu(x)-u(x)|}{\delta}\Big)dx\leq 1,
	$$
	which gives
	$$
	\|\tau_hu-u\|_{L_M(\Omega)}\leq \delta \mbox{ whenever } |h|<\eta.
	$$
\end{proof}
\begin{remark}\label{remark1}
	Note that the boundedness of the function $u$ in Lemma \ref{lem6.1} is necessary, else the result is false. Indeed, when we put ourselves in the particular case $M(x,t)=t^{p(x)}$, the authors \cite{KR} gave the following example : $N=1$, $\Omega=(-1,1)$. For $1\leq r<s<+\infty$ they define the variable exponent
	$$
	p(x)=\left\{
	\begin{array}{lll}
	r & \mbox{ if }x\in[0,1),\\
	s & \mbox{ if }x\in(-1,0)
	\end{array}
	\right.
	$$
	and consider
	$$
	f(x)=\left\{
	\begin{array}{lll}
	x^{-\frac{1}{s}} & \mbox{ if }x\in[0,1),\\
	0 & \mbox{ if }x\in(-1,0).
	\end{array}
	\right.
	$$
	They show that $\tau_{h}f\notin L^{p(\cdot)}(\Omega)$ although that $f\in L^{p(\cdot)}(\Omega)$. Observe here, in this example, that the function $f$ is compactly supported but not bounded on $\Omega$.
\end{remark}
\subsection{Application to approximate identities}
As a consequence of Lemma \ref{lem6.1}, we give the following  density result.
\begin{lemma}\label{lem3.4}
	Let $M\in\Phi$ satisfy (\ref{incBM}) and let $u\in\mathcal{B}_c(\Omega)$. For any $\varepsilon>0$ small enough, we have $u_\varepsilon\in \mathcal{C}^{\infty}_0(\Omega)$. Furthermore,
	$$
	\|u_\varepsilon-u\|_{L_M(\Omega)}\rightarrow 0 \mbox{ as } \varepsilon\rightarrow 0^+.
	$$
\end{lemma}
\begin{proof}
	For $\varepsilon>0$, the function $u_{\varepsilon}$ defined in (\ref{eq11}) belongs to $\mathcal{C}^{\infty}_0(\Omega)$ whenever
	$\varepsilon<dist(supp\;u,\partial\Omega)$ (see for example  \cite[Theorem 2.29]{AA}).
	Let $\overline{M}$ stands for the complementary $\Phi$-function of $M$ and let $v\in L_{\overline{M}}(\Omega)$.
	By Fubini's theorem and H\"{o}lder's inequality (\ref{eq1.1.4}) we can write
	$$
	\begin{array}{lll}\displaystyle
	\int_\Omega\big|\big(u_\varepsilon(x)-u(x)\big)v(x)\big|dx &\leq\displaystyle\int_{\mathrm{R}^N}\Big(\int_\Omega|u(x-\varepsilon y)-u(x)||v(x)|dx\Big)J(y)dy\\
	&\leq\displaystyle2\|v\|_{L_{\overline{M}}(\Omega)}\int_{|y|\leq 1} \|\tau_{-\varepsilon y}u-u\|_{L_M(\Omega)}J(y)dy.
	\end{array}
	$$
	Hence, by the definition of the Orlicz norm and the inequality (\ref{eq1.3.3}) we obtain
	$$
	\|u_\varepsilon-u\|_{L_{M}(\Omega)}\leq2\int_{|y|\leq1}\|\tau_{-\varepsilon y}u-u\|_{L_M(\Omega)}J(y)dy.
	$$
	We can now use Lemma \ref{lem6.1} Given $\mu>0$, there exists $\eta>0$ such that for $\varepsilon<\eta$ we get
	$$
	\|\tau_{-\varepsilon y}u(x)- u(x)\|_{L_M(\Omega)} \leq \mu
	$$
	for every $y$ with $|y|\leq 1$. Then we conclude that
	$$
	\| u_\varepsilon- u\|_{L_{M}(\Omega)}\leq 2\mu \int_{|y|\leq 1}  J(y)dy=2\mu,
	$$
	which gives the result.
\end{proof}
\section{Auxiliary results}\label{sec3}
In the following lemma  we prove that under the condition (\ref{incBM}) bounded functions compactly supported in $\Omega$ are  dense in
$E_M(\Omega)$ with respect to the strong topology and in $L_M(\Omega)$ with respect to the modular topology.
\begin{lemma}\label{th1.1.8}
	Assume that $M\in\phi$ satisfies (\ref{incBM}). Then
	\begin{enumerate}
		\item
		$\mathcal{B}_c(\Omega)$ is dense in $E_M(\Omega)$ with respect to the strong topology in $L_M(\Omega)$.
		\item
		$\mathcal{B}_c(\Omega)$ is dense in $L_M(\Omega)$ with respect to the modular topology in $L_M(\Omega)$.
	\end{enumerate}
\end{lemma}
\begin{proof}
	\begin{enumerate}
		\item [1.] If $u\in E_M(\Omega)$, then  for all $\lambda>0$ one has $M(x,|u|/\lambda)\in L^{1}(\Omega)$. Denote by
	\end{enumerate}
	$T_{j},$ $j>0,$ the truncation function at levels $\pm j$ defined on $\mathrm{R}$ by $T_{j}(s)=\max\{-j,\min\{j,s\}\}.$ We define the sequence $\{u_j\}$ by
	\begin{equation}\label{eq1.1.9}
	u_j=T_{j}(u)\chi_{K_j},
	\end{equation}
	where $\chi_{K_j}$ stands for the characteristic function of the set
	$$
	K_j=\bigg\{x\in \Omega: |x|\leq j, dist(x,\Omega^c)\geq \frac{1}{j}\bigg\}.
	$$
	Hence, the function $u_j\in \mathcal{B}_c(\Omega)$ and converges almost everywhere to $u$ in $\Omega$. Thus $M(x,|u_j(x)-u(x)|/\lambda)\rightarrow 0$ a.e. in $\Omega$ and
	\begin{equation}\label{eq1.7.6}
	M(x,|u_j(x)-u(x)|/2\lambda)\leq M(x,|u(x)|/\lambda) \in L^1(\Omega).
	\end{equation}
	So that by the Lebesgue dominated convergence theorem, we obtain
	$$
	\int_\Omega M(x|u_j(x)-u(x)|/2\lambda)dx\leq 1 \mbox{ for }j \mbox{  large enough},
	$$
	which yields
	$\lim_{j\to+\infty}\|u_j-u\|_{L_M(\Omega)}\leq\lambda$.
	Being $\lambda>0$ arbitrary, we get $\lim_{j\rightarrow +\infty} \|u_j-u\|_{L_M(\Omega)}=0.$
	\begin{enumerate}
		\item [2.] Now if $u\in L_M(\Omega)$ then  for some $\lambda>0$ one has $M(x,|u|/\lambda)\in L^{1}(\Omega)$.
	\end{enumerate}
	Let $\{u_j\}$  the sequence defined in (\ref{eq1.1.9}). The inequality (\ref{eq1.7.6}) holds for some $\lambda>0$ and since $u_j\in \mathcal{B}_c(\Omega)$ and converges a.e. to $u$ in $\Omega$, we get $M(x,|u_j(x)-u(x)|/2\lambda)\rightarrow 0$ a.e. in $\Omega$. Thus, Lebesgue's dominated convergence theorem yields
	$$
	\int_\Omega M(x,|u_j(x)-u(x)|/2\lambda)\rightarrow 0, \mbox{ as } j\rightarrow \infty.
	$$
\end{proof}
\begin{remark}
	In view of Lemma \ref{delta2equivalence}, observe that if the $\phi$-function $M\in\Delta_2$ then bounded functions compactly supported in $\Omega$ are dense in $L_{M}(\Omega)$ for the norm topology.
\end{remark}
\par We denote by $\mathcal{S}$ the family of finite linear combinations of characteristic functions of measurable sets $B_i$ with finite Lebesgue measure, expressed as  follows
$$
\sum_{i=1}^{p}\alpha_i\chi_{B_i}(x), \; \mbox{ with }\alpha_1,\alpha_2,\cdots,\alpha_p\in \mathrm{R} \mbox{ and }|B_i|<+\infty.
$$
Let $\mathcal{S}_c$ stands for the set of the functions belonging to $\mathcal{S}$ with the additional property that
$\cup_{i=1}^{p} B_i\subset K$,
for some compact subset $K$ of $\Omega$. In the following lemma  we prove the density of $\mathcal{S}_c$ in $E_M(\Omega)$.
\begin{lemma}\label{th1.2.1}
	Assume that $M\in\phi$ satisfies (\ref{incBM}). Then  the set $\mathcal{S}_c$ is dense in $E_M(\Omega)$ with respect to the strong topology in $E_M(\Omega)$.
\end{lemma}
\begin{proof}
	Let $u\in\mathcal{B}_c(\Omega)$. Since $u$ is a measurable function, by classical result \cite{AA} we know that there exists a sequence $\{u_n\}\subset\mathcal{S}$ converging pointwise to $u$ in $\Omega$ and satisfying $|u_n(x)|\leq |u(x)|$, for all $n\in\mathrm{N}$ and $x\in\Omega$. Since $u\in\mathcal{B}_c(\Omega)$, we can assume that $u_n\in\mathcal{S}_c$. Hence,
	$$
	M(x,|u_n(x)-u(x)|/\lambda) \leq M(x,2|u(x)|/\lambda) \in L^1(\Omega).
	$$
	As for all $\lambda>0$
	$$
	M(x,|u_n(x)-u(x)|/\lambda)\rightarrow 0\mbox{ a.e. in } \Omega,
	$$
	by the Lebesgue dominated convergence theorem we obtain
	$$
	\int_\Omega M(x,|u_n(x)-u(x)|/\lambda)dx\leq1, \mbox{ for } n \mbox{ large enough},
	$$
	which yields $\|u_n-u\|_{L_M(\Omega)}\leq\lambda$, for $n$  large enough. Being $\lambda>0$ arbitrary,  we get
	$$
	\|u_n-u\|_{L_M(\Omega)}\rightarrow0 \mbox{ as } n\rightarrow \infty.
	$$
\end{proof}

\begin{lemma}
	Assume that $M\in\Phi$ satisfies (\ref{incBM}). For every nonempty subset $E\subset K$ where $K$ is a compact subset of $\Omega$, there exist two constant numbers $c_1$, $c_2\geq 0$ such that
	\begin{equation}\label{eq1.6.2}
	\|\chi_{E}\|_{L_M(\Omega)}\leq \frac{1}{M^{-1}(c_1,\frac{c_2}{|E|})}.
	\end{equation}
\end{lemma}
\begin{proof}
	Let $x_0\in\Omega$ be fixed. By the assumption (\ref{incBM}), the measurable function $x\mapsto M\Big(x,M^{-1}(x_0,\frac{1}{2|E|})\chi_{E}\Big)$ belongs to $L^1(\Omega)$. Hence, there is an $\eta>0$ such that for any measurable subset $\Omega^{\prime}$ of $\Omega$, one has
	$$
	|\Omega^{\prime}|<\eta\Rightarrow \int_{\Omega^{\prime}}M\Big(x,M^{-1}(x_0,\frac{1}{2|E|})\chi_{E}\Big)dx<\frac{1}{2}.
	$$
	As $M(\cdot,s)$ is measurable on $E$, Luzin's theorem implies that for $\eta>0$ there  exists a closed set $F_{\eta}\subset E$ such that the restriction
	of $M(\cdot,s)$ to $F_{\eta}$ is continuous and $|E\setminus F_{\eta}|<\eta$.
	Let $k$ be the point where the supremum of $M(\cdot,s)$ is reached in the set $F_{\eta}$.
	$$
	\begin{array}{lll}\displaystyle
	\int_{E} M\Big(x,M^{-1}(k,\frac{1}{2|E|})\Big)dx&=\displaystyle\int_{F_{\eta}}M\Big(x,M^{-1}(k,\frac{1}{2|E|})\Big)dx\\
	&+\displaystyle\int_{E\setminus F_{\eta}}M\Big(x,M^{-1}(k,\frac{1}{2|E|})\Big)dx.
	\end{array}
	$$
	For the first term in the right-hand side of the last equality, we use   (\ref{inversefunction1}) obtaining
	$$
	\int_{F_{\eta}} M\Big(x,M^{-1}(k,\frac{1}{2|E|})\Big)dx\leq\int_{F_{\eta}} M\Big(k,M^{-1}(k,\frac{1}{2|E|})\Big)dx \leq \frac{1}{2},
	$$
	while for the second one, since $|E\setminus F_{\eta}|<\eta$ we have
	$$
	\int_{E\setminus F_{\eta}}M\Big(x,M^{-1}(k,\frac{1}{2|E|})\Big)dx\leq \frac{1}{2}.
	$$
	Thus, we get
	$$
	\int_{\Omega} M\Big(x,M^{-1}(k,\frac{1}{2|E|})\chi_{E}\Big)dx
	\leq 1.
	$$
\end{proof}
\begin{remark}
	Formula (\ref{eq1.6.2}) remains valid either for every nonempty bounded subset $E$ of $\Omega$ or for every nonempty subset $E$ of $\Omega$ if $\Omega$ is a bounded open.
\end{remark}
\begin{lemma}\label{lem6.2}
	Let $M$ and $\overline{M}$ be two complementary $\Phi$-functions. Let $v\in L_{\overline{M}}(\Omega)$ be a fixed function. Define
	\begin{equation}\label{eq1.5.4}
	L_v(u)=\int_\Omega u(x)v(x)dx,\; \forall u\in L_M(\Omega).
	\end{equation}
	Then $L_v$ defines a linear continuous functional on $L_M(\Omega)$. Furthermore,
	$$
	\|v\|_{L_{\overline{M}}(\Omega)}\leq\|L_v\|\leq 2\|v\|_{L_{\overline{M}}(\Omega)},
	$$
	where $\|L_v\|=\sup\{|L_v(u)|, \|u\|_{L_M(\Omega)}\leq 1\}$.
\end{lemma}
\begin{proof}
	We omit the proof, since it's similar to the one given in \cite[Lemma 8.17]{AA} in the framework of Orlicz spaces.
	\par The above result holds also when $L_v$ is restricted to $E_{M}(\Omega)$. In general, continuous linear functionals defined on $L_M(\Omega)$ can be
	expressed in a different form to that defined in (\ref{eq1.5.4}) (see \cite[Theorem 3.13.5]{KJF}). Hence, we can not have the Riesz representation
	theorem  as in classical Lebesgue spaces. In the following lemma we give an "almost complete" analogue of Riesz representation theorem. A similar result
	in the Orlicz framework can be found in \cite[Theorem 3.13.6]{KJF}.
\end{proof}
\begin{lemma}\label{lem1.1.9}
	Assume that $M\in \Phi$ satisfies (\ref{incBM}) and let $L\in [E_M(\Omega)]^{\prime}$. Then there exists a unique function $v\in L_{\overline{M}}(\Omega)$ such that
	\begin{equation}\label{eq1.6.4}
	L(u)=\int_\Omega u(x)v(x)dx,~~ \forall u\in E_M(\Omega).
	\end{equation}
\end{lemma}
\begin{proof}
	We begin first by assuming that $u$ belongs to $\mathcal{S}_c$ i.e. $u$ is of the form  $\sum_{i=1}^{p}\alpha_i\chi_{B_i}(x)$ where $B_i$
	are measurable sets of finite Lebesgue measure such that $\cup_{i=1}^{p} B_i\subset K$ for some compact subset $K\subset\Omega$ and $\alpha_i\in\mathrm{R}$, for $i=1,2,\cdots,p$. Let $\mu$ be the complex measure defined for a measurable set of finite Lebesgue measure $A\subset K\subset\Omega$ for some compact $K$ as follows
	$$
	\mu(A)=L(\chi_{A}).
	$$
	By (\ref{eq1.6.2}) there exist two constants $c_1,$ $c_2\geq 0$ such that
	$$
	|\mu(A)|\leq  \|L\|\|\chi_{A}\|_{L_M(\Omega)}\leq  \frac{\|L\|}{M^{-1}(c_1,c_2/|A|)}\rightarrow 0\mbox{ as } |A|\rightarrow 0.
	$$
	Thus, the measure $\mu$ is absolutely continuous with respect
	to the Lebesgue measure and it follows by Radon-Nikodym's theorem that there exists a nonnegative measurable function $v\in L^1(\Omega)$, unique up to sets of Lebesgue measure zero, such that
	$$
	\mu(A)=\int_{A}v(x)dx.
	$$
	Hence, we can write
	\begin{equation}\label{eq1.6.3}
	\begin{array}{lll}
	L(u)&=\displaystyle\sum_{i=1}^p\alpha_iL(\chi_{B_i})=\sum_{i=1}^p\alpha_i\mu(B_i)
	=\sum_{i=1}^p\alpha_i\int_{B_i}v(x)dx\\
	&=\displaystyle\sum_{i=1}^p\alpha_i\int_\Omega v(x)\chi_{B_i}dx=\int_\Omega u(x)v(x)dx.
	\end{array}
	\end{equation}
	Now for arbitrary $u\in E_M(\Omega)$, by Lemma \ref{th1.2.1} we can found a sequence of functions $u_j\in \mathcal{S}_c$ such that $u_j\rightarrow u$ a.e. in $\Omega$ and strongly in $E_M(\Omega)$. Therefore, by Fatou's lemma we obtain
	$$
	\begin{array}{lll}\displaystyle
	\big|\int_\Omega u(x)v(x)dx\big|&\leq\displaystyle\liminf_{j\to+\infty}\int_\Omega |u_j(x)v(x)|dx=\liminf_{j\to+\infty}L(|u_j|\mbox{sgn}v)\\
	&\leq\displaystyle\|L\|\liminf_{j\to+\infty}\|u_j\|_{L_M(\Omega)}\leq\|L\|\|u\|_{L_M(\Omega)}.
	\end{array}
	$$
	This implies that $v\in L_{\overline{M}}(\Omega)$. Let $L_v(u)=\displaystyle\int_\Omega u(x)v(x)dx$, the linear functional defined by (\ref{eq1.5.4}). By (\ref{eq1.6.4}), $L_v$ and $L$ coincide on the set $\mathcal{S}_c$ and by Lemma \ref{th1.2.1}, they coincide everywhere in $E_M(\Omega)$.
\end{proof}
\section{Proof of the main results}\label{sec4}
\subsubsection*{Proof of Theorem \ref{th1.1.1}}
\begin{enumerate}
	\item Combining Lemma \ref{lem3.4} and Lemma \ref{th1.1.8}, we obtain the density of $\mathcal{C}^{\infty}_{0}(\Omega)$ in $E_M(\Omega)$ with respect to the strong topology.
	\item Let $u\in L_M(\Omega)$. According to Lemma \ref{th1.1.8}, there exist $w\in \mathcal{B}_c(\Omega)$ and $\lambda>0$ such that for all $\eta\geq0$
	$$
	\int_{\Omega}M(x,|u(x)-w(x)|/\lambda)dx\leq\eta.
	$$
	Then by Lemma \ref{lem3.4}, there exists a function $v\in \mathcal{C}^\infty_0(\Omega)$ that converges strongly to $w$ in $L_M(\Omega)$. But we know that the norm topology is strong than the
	modular one, more precisely we have
	$$
	\int_{\Omega}M(x,|w(x)-v(x)|)dx\leq\eta.
	$$
	Let us make the choice $\lambda_{1}=\max\{1,\lambda\}$ and use the convexity of the $\Phi$-function $M$, we can write
	$$
	\begin{array}{lll}\displaystyle
	\int_\Omega M(x,|u(x)-v(x)|/2\lambda_1)dx&\leq\displaystyle\frac{1}{2}\int_\Omega M(x,|u(x)-w(x)|/\lambda)dx\\
	&+\displaystyle\frac{1}{2}\int_\Omega M(x,|w(x)-v(x)|)dx.
	\end{array}
	$$
	This, yields the result.
\end{enumerate}
\begin{remark}\label{rem1.1.3}
	Let $M\in\Phi$. If $M\in\Delta_2$, then $\mathcal{C}^\infty_0(\Omega)$ is dense in $L_M(\Omega)$ with respect to the norm $\|\cdot\|_{L_M(\Omega)}$.
\end{remark}
\subsubsection*{Proof of theorem \ref{th1.1.2}}
Let $u\in E_M(\Omega)$. By virtue of Theorem \ref{th1.1.1}, we can assume that $u\in \mathcal{C}_c(\Omega)$ (the set of continuous functions with compact support in $\Omega$). Hence, it's sufficient to show that there exists a countable set dense in $\mathcal{C}_c(\Omega)$ with respect to the strong topology in $E_M(\Omega)$.
Let $\Omega_n$ be the sequence of compact subsets of $\mathrm{R}^N$ defined by
$$
\Omega_n=\{x\in \Omega: |x|\leq n \mbox{ and } dist(x,\partial \Omega)\geq \frac{1}{n}\}.
$$
Recall that
$\Omega=\cup_{i=1}^{+\infty}\Omega_n$.
Let $\mathcal{P}$ be the set of all polynomials on $\mathrm{R}^N$ with rational coefficients and
$$
\mathcal{P}_n=\{v\chi_{\Omega_n}: v\in \mathcal{P}\},
$$
where $\chi_{\Omega_n}$ is the characteristic function of $\Omega_n$.  If $dist(supp\; u,\partial \Omega)>\frac{1}{n}$ then $u$ belongs to $\mathcal{C}(\Omega_n)$ and by using the density of  $\mathcal{P}_n$ in $\mathcal{C}(\Omega_n)$ \cite[corollary 1.32]{AA},  there exists a sequence $u_j\in\mathcal{ P}_n$ that converges uniformly to $u$, that is to say
$$
\forall \varepsilon>0, \exists j_0\in\mathrm{N}, \forall j\geq j_0, \sup_{x\in \Omega_n}|u_j(x)-u(x)|\leq \frac{\varepsilon}{\displaystyle\int_{\Omega_n} M(x,1)dx+1}.
$$
So that for every $\varepsilon>0$, there exists $j_0\in\mathrm{N}$ such that for any $j\geq j_0$, one has
$$
\begin{array}{lll}\displaystyle
\int_{\Omega_n} M\Big(x,\frac{|u(x)-u_j(x)|}{\varepsilon}\Big)dx\leq 1.
\end{array}
$$
Therefore, $\mathcal{P}_n$ is dense in $\mathcal{C}(\Omega_n)$ for the strong topology in $L_M(\Omega)$. Consequently, the countable set
$\cup_{n=1}^\infty \mathcal{P}_n$ is dense in $E_M(\Omega)$ and $E_M(\Omega)$ is separable.
\begin{remark}\label{rem1.1.1}
	In view of Lemma \ref{delta2equivalence}, if the $\Phi$-function $M\in\Delta_2$ and satisfies (\ref{incBM}) then $L_M(\Omega)$ is a separable space.
\end{remark}
\subsubsection*{Proof of theorem \ref{thE1}}
It is immediate that from Lemma \ref{lem1.1.9}, we get the isomorphism
$$
L_{\overline{M}}(\Omega)\simeq [E_M(\Omega)]^{\prime}.
$$
\subsection*{Proof of theorem \ref{th1.2.2} }
Since $M$ and $\overline{M}$ satisfy both the $\Delta_2$-condition, by
Lemma \ref{delta2equivalence} and Theorem \ref{thE1} we get
$$
L_{\overline{M}}(\Omega)\simeq[L_M(\Omega)]^\prime \mbox{ and }L_{M}(\Omega)\simeq[L_{\overline{M}}(\Omega)]^\prime
$$
and then the conclusion follows.




\end{document}